\documentclass[12pt]{amsart}
\usepackage[left=2.3cm, hmarginratio=1:1, includemp, marginparwidth=2cm, bottom=1.5cm, foot=.7cm, includefoot, textheight=22cm,textwidth = 16.5cm]{geometry}
\usepackage{amsmath, amssymb, amsthm, mathrsfs, mathtools, amscd}
\usepackage{graphicx, longtable, afterpage, float}
\usepackage{xspace}
\usepackage[latin1]{inputenc}
\usepackage[usenames,dvipsnames]{color}
\usepackage{rotating}
\usepackage{tikz}
\usetikzlibrary{arrows,shapes}
\usepackage{setspace, yfonts}
\usepackage{epsfig,cancel,cite}

\usepackage[all]{xy}

\renewenvironment{proof}{{\noindent \bfseries Proof:}}{}
\newtheorem{theorem}{Theorem}[section]
\newtheorem{lemma}[theorem]{Lemma}

\newtheorem{proposition}[theorem]{Proposition}
\newtheorem{definition}[theorem]{Definition}
\newtheorem{remark}[theorem]{Remark}

\renewcommand{\qed}{\hfill \ensuremath{\Box}}

\author{Carmen Maria Constantin, Andreas D\"oring}
\title{A Topos Theoretic Notion of Entropy}

\begin{document}

\begin{abstract}
In the topos approach to quantum theory, the spectral presheaf plays the role of the state space of a quantum system. We show how a notion of entropy can be defined within the topos formalism using the equivalence between states and measures on the spectral presheaf. We show how this construction unifies Shannon and von Neumann entropy as well as classical and quantum Renyi entropies. The main result is that from the knowledge of the contextual entropy of a quantum state of a finite-dimensional system, one can (mathematically) reconstruct the quantum state, i.e., the density matrix, if the Hilbert space is of dimension $3$ or greater. We present an explicit algorithm for this state reconstruction and relate our result to Gleason's theorem.
\end{abstract}

\maketitle
\section{Introduction}

It has been argued that Quantum Mechanics can be understood in a more natural way as a theory about the possibilities and impossibilities of information transfer and processing, as opposed to a theory about the mechanics of nonclassical waves or particles \cite{1,2,3,5}. Therefore understanding the representation and manipulation of information can help us shed light on the fundamental structure of both classical and quantum theories and it can lead to fresh insights about the essential differences between these two.

The concept of entropy plays an important role within Information Theory. For classical systems, Shannon entropy \cite{Sha48} is typically used, for quantum systems, von Neumann entropy \cite{vN55}. Several generalizations of these have already been considered \cite{stephanie,barnum}. It is interesting to ask how much information about a quantum state can be encoded using Shannon and von Neumann entropies, and in particular how much information can be encoded by simultaneously considering all classical perspectives on a given state and their associated classical entropies.

This leads us to define a new notion of contextual entropy which unifies Shannon and von Neumann entropy using topos theoretic formalism. The ingredients required for this construction were initially introduced in the context of the topos approach to quantum mechanics by Chris Isham and Jeremy Butterfield, and later developed by Chris Isham and Andreas D\"{o}ring. One of the aims of the programme was to introduce a new mathematical framework in which to express quantum mechanics in a way that was structurally similar to classical mechanics. So far at the most basic level, physical theories have been formulated within the mathematical universe of sets and functions. The universe of sets and functions is one example of a topos. However, more general universes (or topoi) do exist and many physical structures such as states, observables and propositions about these, can be formulated in a natural way within these generalised settings. This reformulation hopes to offer a fresh insight into the structural fabric of physical reality.

One of the main ideas underpinning the topos formalism is that one might hope to obtain a complete description of a quantum system by looking at that system from all possible classical perspectives, and keeping track of the information obtained in this way. This is also one of the main ideas behind the present approach to give a contextual definition of entropy: we try to give a definition that takes into account all classical perspectives on a quantum state at the same time. Moreover our definition does not directly depend on the interpretation of states as density matrices within the usual Hilbert space formalism of quantum mechanics, but we define entropy on the set of finitely additive probability measures on a certain non-commutative space, the so-called spectral presheaf. Andreas D\"{o}ring has shown \cite{ms} using Gleason's theorem that this set of measures is in fact equivalent to the set of states. 

We show here that von Neumann entropy, which characterises quantum states up to unitary equivalence, appears within a certain classical context. Moreover, we show that contextual entropy, which is the collection of Shannon entropies associated with all the available classical perspectives on a quantum state, contains enough information to completely reconstruct the quantum state. 

Contextual entropy therefore provides a mathematical encoding of quantum states based on the information theoretical concept of entropy. As such it is a step towards an information-theoretic characterisation of quantum states.

In Section \ref{ctxt} we show how a measure on the spectral presheaf (i.e. a state) gives us a canonical probability distribution in each classical context, and how it is therefore possible to associate a Shannon entropy to each classical 'perspective' on a state. Contextual entropy is defined in terms of this collection of Shannon entropies, which are shown to form a global section of a certain real-number presheaf. We also show how one can retrieve the von Neumann entropy of a state from such a global section. This confirms our expectation that entropy within the topos approach should 'look' like Shannon entropy from each classical perspective, but one should also be able to retrieve the quantum mechanical von Neumann entropy by taking into account all perspectives at the same time. In fact, one can do even more than this, and we show that contextual entropy encodes enough information to explicitly reconstruct the quantum state from which it originated. This argument relies on a powerful result known as the Schur-Horn Lemma. One of its advantages is that it provides us with us a new insight into Gleason's theorem. This is discussed in Section \ref{gls}.

In Section \ref{cns} we make a comparison between Shannon, von Neumann and contextual entropies. This allows us to observe, for example, that one of the differences between Shannon and von Neumann entropies (the property of being monotone) is precisely due to contextuality, although this idea is not explicitly taken into account into the definition of these two entropies. 

Finally, in Section \ref{other} we show that it is possible to adapt other classical entropies within the formalism of the topos approach, given that they satisfy a certain weak recursivity property. We will show how Renyi entropies can be defined within the topos formalism, and moreover we will see that contextual Renyi entropies also encode sufficient information to allow for state reconstruction.

\section{Background}\label{bckg}
\subsection{The spectral presheaf}\label{spectral}

The spectral presheaf is a central object in the topos approach to quantum theory. It was introduced by Isham and Butterfield \cite{b1,b2,b3,b4} and later used by Isham and D\"{o}ring \cite{d1,d2,d3,d4}. Within the topos approach, the spectral presheaf associated with the von Neumann algebra of physical quantities/observables of a quantum system is the analogue of the state space of a classical system. The idea behind this construction is that we can hope to obtain a complete picture of a quantum system (described by a non-commutative von Neumann algebra $N$) by fitting together all the classical perspectives on that system in a consistent way.

We call a commutative subalgebra of $N$ a \textit{context} and we assume that every context contains the identity operator $1$. Von Neuman's double commutant construction establishes a bijective correspondence between commutative subalgebras of a type I von Neumann algebra and pairwise orthogonal families of projections which add up to the identity. Such families are usually used in quantum mechanics to describe projective measurements. Thus every commutative subalgebra $V$ can be expressed in a cannonical way as $V=\{P_1,\ldots,P_n\}''$. 

Therefore a context can be interpreted as being the family of projections onto eigenspaces of any observable of the form $A=\sum_{i=1}^k a_i P_i$.  Observables that can be measured simultaneously are represented by self-adjoint operators that can be diagonalized simultaneously, hence they have a joint set of eigenspaces. A context corresponds precisely to such a collection of eigenspaces and as such it corresponds to a family of simultaneously measurable observables. This allows us to also interpret a context as a `classical perspective' on the quantum system. 

The set of all contexts, minus the trivial one: $V_0=\mathbb{C}1$, is denoted by $\mathcal{V}(N)$. This is a partially ordered set under inclusion, and as such it forms a category. The objects of the context category are the contexts themselves, while the arrows are the inclusion maps.

\begin{definition}
The \textbf{spectral presheaf} $\underline{\Sigma}^N$ of a given von Neumann algebra $N$ is the following contravariant functor from the category $\mathcal{V}(N)$ to the category of sets:

\begin{enumerate}
\item[a)] on objects: for all $V\in \mathcal{V}(N)$, let $\underline{\Sigma}^N_V$ be the Gelfand spectrum of $V$, i.e. the set of multiplicative positive linear functionals of norm one, or equivalently, the set of pure states on $V$, equipped with the weak-* topology
\item[b)] on arrows: for all inclusions $i_{VV'}:V'\hookrightarrow V$, let $\underline{\Sigma}^N(i_{VV'}): \underline{\Sigma}^N_V \rightarrow \underline{\Sigma}^N_{V'}$ be the function that sends each pure state $f$ to its restriction $f|_{V'}$ to the
smaller algebra. This function is well-known to be continuous and surjective.
\end{enumerate}
\end{definition}

When no confusion arises we will simply write $\underline{\Sigma}$ instead of $\underline{\Sigma}^N$.

The projections of a von Neumann algebra $N$ stand for propositions of the form "$A\varepsilon\Delta$", that is propositions of the form "the physical quantity $A$, which is
represented by the self-adjoint operator $A\in N$, has a value in the set Borel set $\Delta$". More precisely, each projection corresponds to an equivalence class of such
propositions.

If we take a commutative subalgebra $V$ of $N$, every state $f\in \mathrm{Spec}\, V$ of $V$ gives us a way to assign truth values to propositions which involve quantities
represented by self-adjoint operators from $V$. Any such $f$ can take only one of the two values $0,1$ when applied to a projection $P\in V$, since
$$f(P)=f(P^2)=f(P)f(P)$$
So we can assign to those propositions which correspond to the projection $P$ the value true if $f(P)=1$, and false if $f(P)=0$. We know from the Kochen-Specker theorem that it
would not possible to make such truth-value assignments for the projections of the non-commutative algebra $N$ (unless $N$ was a type $I_2$-algebra).

The projections in a commutative von Neumann algebra $V$ correspond bijectively to clopen subsets of $\mathrm{Spec}\, V$:

\begin{proposition}\label{alpha}
If $\mathcal{P}(V)$ is the lattice of all projections in $V$ and $Cl(\mathrm{Spec}\, V)$ is the lattice of clopen subsets of $\mathrm{Spec}\, V$, then the map 
\begin{align*}
\alpha_V: \mathcal{P}(V)&\rightarrow Cl(\mathrm{Spec}\, V) \\
P&\mapsto S_{P}:=\{f\in \mathrm{Spec}\, V~|~f(P)=1\}
\end{align*}
is a lattice isomorphism.
\end{proposition}

\begin{proof}
It is easy to check that $S_{P}$ is indeed a clopen subset of $\mathrm{Spec}\, V$. We have
$$S=\overline{P}^{-1}\left(\,\left(\frac{1}{2},\infty\right)\,\right)$$
and so $S$ is open. Similarly 
$$\mathrm{Spec}\, V \backslash S= \overline{P}^{-1}\left(\,\left(-\infty,\frac{1}{2}\right)\,\right)$$
and so $\mathrm{Spec}\, V\backslash S$ is open, hence $S$ is closed.

Since the Gelfand representation is a *-isomorphism for unital commutative algebras, $\alpha_V$ must be a bijective map. \qed
\end{proof}

We have seen that for a `classical part' of a quantum system described by a commutative algebra $V$ there is a correspondence between propositions, or rather the projections which
represent them, and clopen subsets of the Gelfand spectrum of the algebra $V$. Next we will see that for quantum systems (as a whole) there is an analogous correspondence between
propositions and clopen sub-objects of the spectral presheaf.

The collection of all contexts of a non-commutative von Neumann algebra $N$ can be understood as the collection of all classical perspectives on a quantum system. As we have
mentioned before, the idea behind the spectral presheaf is to characterise a quantum system by taking into account all the classical perspectives at the same time. In order to do
this, we need to adapt every proposition about the whole quantum system to each possible classical context. That is, given a proposition "$A\varepsilon\Delta$" and its representing projection $P$, we want to choose for every context $V$ the strongest proposition implied by "$A\varepsilon\Delta$" which can be made from the perspective of that context. For
projections, this is equivalent to taking the smallest projection in any context $V$ that is larger or equal to $P$:
$$\delta^o(P)_V:=\bigwedge\{ Q\in\mathcal{P}(V)~|~Q\geq P\}$$
If $P\in\mathcal{P}(V)$, the above approximation will simply be equal to $P$. We will call the original proposition "$A\varepsilon\Delta$" the \textit{global proposition}, while a
proposition "$B\varepsilon\Gamma$" corresponding to the projection $\delta^o(P)_V$ will be called a \textit{local proposition}.

From the family of projections $(\delta^o(P)_V)_{V\in\mathcal{V}(N)}$ we can obtain a family of clopen subsets of the Gelfand spectra $(\mathrm{Spec}\, V)_{V\in\mathcal{V}(N)}$ by
choosing for every $V$ the subset
$$S_{\delta^o(P)_V}=\alpha_V(\delta^o(P)_V)\subseteq\mathrm{Spec}\, V$$

These subsets behave nicely under the restriction mappings of the spectral presheaf $\underline{\Sigma}$ and so we can give the following definition.

\begin{definition}
The \textbf{daseinisation of a projection $P$} is the subobject (or equivalently, the subpresheaf) $\underline{\delta(P)}$ of the spectral presheaf $\underline{\Sigma}$ given by
the collection $(S_{\delta^o(P)_V})_{V\in\mathcal{V}(N)}$ of clopen subsets, together with the restriction mappings between them.
\end{definition}

The daseinisation $\underline{\delta(P)}$ of a projection $P$ representing the proposition "$A\varepsilon\Delta$" can be seen as the analogue of the measurable subset
$f_A^{-1}(\Delta)$ of the state space of a classical system. We say that $\underline{\delta(P)}$ is the representative of the global proposition "$A\varepsilon\Delta$".

The daseinisation $\underline{\delta(P_\psi)}$ of a projection $P_\psi$ which projects onto the ray spanned by the vector $\psi$ is called the \textit{pseudo-state} associated to
$\psi$. It is the analogue of a point in the state space of a classical system. It is important to note that the pseudo-states are not global elements of $\underline{\Sigma}$. In
fact, global elements of a presheaf are the category-theoretical analogues of points. Isham and Butterfield have observed \cite{b1} that the Kochen-Specker theorem is equivalent to the fact that the spectral presheaf has no global elements. A global element $\gamma$ of $\underline\Sigma$ would pick one $\gamma_V\in\underline\Sigma_V$ for each context $V$ such that, whenever $V'\subset V$, one would have $\gamma_V|_{V'}=\gamma_{V'}$. Each $\gamma_V$ assigns values to all physical quantities described by self-adjoint operators $A$ in $V$
by evaluation, i.e., by simply forming $\gamma_V(A)$. If $A$ is contained in different commutative subalgebras $V,\widetilde V$, then it is also contained in $V':=V\cap\widetilde
V$, and $\gamma_V(A)=\gamma_{V'}(A)=\gamma_{\widetilde V}(A)$, so the defining condition of the global element $\gamma$ guarantees that $A$ is assigned the same value in every
context. Since every self-adjoint operator is contained in some commutative subalgebra $V$, a global element $\gamma$ of $\underline\Sigma$ would provide a consistent assignment
of values to all self-adjoint operators. But the Kochen-Specker theorem precisely shows that this is impossible, hence such global elements $\gamma$ cannot exist.

Pseudo-states however are minimal sub-objects in a suitable sense: they come from rank-1 projections, the smallest non-trivial projections, and daseinisation is order-preserving,
so pseudo-states are the smallest non-trivial sub-objects of $\underline\Sigma$ that can be obtained from daseinisation. Hence, pseudo-states are `as close to points as possible'.

\begin{definition}
A subobject $\underline{S}$ of the spectral presheaf $\underline{\Sigma}$ such that for each $V\in\mathcal{V}(N)$ the component $\underline{S}_V$ is a clopen subset of
$\underline{\Sigma}_V$ is called a \textbf{clopen subobject}.
\end{definition}

Note that all sub-objects obtained from the daseinisation of projections are clopen. The sub-objects of the spectral presheaf are the quantum analogues of subsets of the phase space in classical physics. The collection of all sub-objects of the spectral presheaf can be turned into a complete Heyting algebra (and hence also a frame) by defining suitable meet and join operations. 

\begin{definition}
 If $\underline{S_1}$ and $\underline{S_2}$ are two sub-objects of the spectral presheaf, their join is defined by stagewise unions in the following way:
$$(\underline{S_1}\vee\underline{S_2})_V=\underline{S_1}_V\cup\underline{S_2}_V$$
Similarly, their meet is given by stagewise intersections:
$$(\underline{S_1}\wedge\underline{S_2})_V=\underline{S_1}_V\cap\underline{S_2}_V$$
\end{definition}

We denote the collection of all sub-objects of the spectral presheaf by $Sub(\underline{\Sigma})$. It can be seen as the analogue of the power set of the state space of a classical system.

One can prove that the clopen sub-objects of the spectral presheaf $\underline{\Sigma}$ also form a complete Heyting algebra under stagewise meet and join operations. We will denote this algebra by $Sub_{cl}(\underline{\Sigma})$. This can be seen as the analogue of the collection of measurable subsets of the state space of a classical system.

\subsection{States as measures on the spectral presheaf}\label{measures}

In general, a state on a von Neumann algebra is a positive linear functional of unit norm on that algebra. Given such a state, we can associate to it a certain measure on the
corresponding spectral presheaf. This construction was explored in detail by D\"oring, who also showed that measures on the spectral presheaf can be defined without reference to
states and moreover that from each abstractly defined measure a unique state can be reconstructed \cite{ms}. We give a brief overview of these ideas below. 

In classical physics states are represented by probability measures on state space, and pure states are represented by Dirac measures. A probability measure assigns a number between $0$ and $1$ to each measurable subset of state space. Within the topos approach the role of the state space is played by the spectral presheaf, and so in analogy with classical mechanics we would like states to be represented by probability measures on (clopen subobjects of) the spectral presheaf. However, since subobjects of the spectral presheaf are not simply sets, but collections of sets, we can not expect the values taken by the measure to be given by single numbers. Instead we would expect to obtain a collection of such numbers, one for each context of the algebra which represents our system. With this in mind we give the following definitions of real number objects and their global sections, which will be essential to our discussion.

\begin{definition}
Given a von Neumann algebra $N$ and its associated poset of abelian subalgebras $\mathcal{V}(N)$, let $\downarrow V:=\{W\in\mathcal{V}(N)~|~W\subseteq V\}$ denote the down-set of a context $V\in\mathcal{V}(N)$.The presheaf $\underline{\mathbb{R}^\succeq}$ is defined
\begin{itemize}
\item on objects: $\underline{\mathbb{R}^\succeq}_V=\{f:\downarrow V\rightarrow \mathbb{R}~|~ f \text{ is order reversing }\}$
\item on arrows: for $i_{V'V}: V' \hookrightarrow V$, $\underline{\mathbb{R}^\succeq}(i_{V'V}):\mathbb{R}\rightarrow\mathbb{R}$ is given by $$\underline{\mathbb{R}^\succeq}(i_{V'V})(f):=f|_{_{\downarrow V'}}$$
\end{itemize}
Note that this presheaf lives in the same topos as $\underline{\Sigma}^N$. However, we do not explicitly specify this topos, by indicating the base category, when discussing this and similar real-number presheaves. It is usually clear from the context, which base category we are using.

A global section of this presheaf can be regarded as an order-reversing function from the partially ordered set $\mathcal{V}(N)$ to the real numbers equipped with the usual ordering.
\end{definition}

The presheaf defined above plays an important role within the topos approach, and is discussed extensively in \cite{d3,disham}. However, when defining measures we will only use a sub-presheaf of this real-number presheaf, which we denote by $\underline{[0,1]^\succeq}$. Later on, when we will introduce the notion of entropy we will encounter a closely related presheaf, $\underline{[0.\ln n]^\preceq}$, where $n$ (this time finite) denotes the dimension of the algebra corresponding to our system. In this case global sections will be equivalent to order-preserving functions from $\mathcal{V}(N)$ to the real number interval $[0,\ln n]$.

\begin{definition}
Given a von Neumann algebra $N$, a measure on its associated spectral presheaf $\underline{\Sigma}$ is a mapping
\begin{align*}
 \mu: \mathrm{Sub}_{\mathrm{cl}}(\underline{\Sigma}) & \longrightarrow \Gamma \underline{[0,1]^{\succeq}}\\
 \underline{S}=(\underline{S}_V)_{V\in\mathcal{V}(N)} &\longmapsto  \mu(\underline{S}):\mathcal{V}(N)\rightarrow [0,1] \\
 & \ \ \ \ \ \ \ \ \ \ \ \ \ \ \ \ V\ \ \ \mapsto  \mu(\underline{S}_V) 
\end{align*}
which satisfies the following conditions:
\begin{enumerate}
 \item $\mu(\underline{\Sigma})=1_{\mathcal{V}(N)}$
 \item for all $\underline{S}_1$, $\underline{S}_2\in\mathrm{Sub}_{\mathrm{cl}}(\underline{\Sigma})$, it holds that 
  $$\mu(\underline{S}_1\vee\underline{S}_2)+\mu(\underline{S}_1\wedge\underline{S}_2)=\mu(\underline{S}_1)+\mu(\underline{S}_2)$$
\end{enumerate}
where the addition, just like the meet and the join for sub-objects, is defined as a stagewise operation.

These conditions also imply that $\mu(\underline{0})=0$, where $\underline{0}$ is the subobject of $\underline{\Sigma}$ which assigns the empty set to each context.
\end{definition}
Note we have abused notation slightly in the above definition by writing $\mu$ both for the measure and for its contextual components.

In this text we will mostly be concerned with a particular type of von Neumann algebras, the algebras of bounded linear operators on finite dimensional Hilbert
spaces (i.e. matrix algebras).  For these algebras the unit norm positive linear functionals can be identified with the density matrices: to each density matrix $\rho\in M_n$, we can associate the functional
\begin{align*}
A &\longmapsto \mathrm{Tr}(\rho A), \ \ \  \forall A\in M_n
\end{align*}
and moreover every positive linear functional of unit norm is of this form in the finite dimensional setting. With this in mind, when talking about matrix algebras we shall refer to the density matrices as states on those algebras.

\begin{definition}
Given a state $\rho$ on the matrix algebra $M_n$, it is straightforward to define its associated measure:
\begin{align*}
 \mu_\rho: \mathrm{Sub}_{\mathrm{cl}}(\underline{\Sigma}^{M_n}) & \longrightarrow \Gamma \underline{[0,1]^{\succeq}}\\
 \underline{S}=(\underline{S}_V)_{V\in\mathcal{V}(M_n)} &\longmapsto  \mu_\rho(\underline{S}):\mathcal{V}(M_n)\rightarrow [0,1] \\
 & \ \ \ \ \ \ \ \ \ \ \ \ \ \ \ \ \ \ \ V\ \ \ \ \mapsto  \mathrm{Tr}(\rho P_{\underline{S}_V})
\end{align*}
where $P_{\underline{S}_V}=\alpha_V^{-1}(\underline{S}_V)$. 
\end{definition}

One can easily check that the function $\mu_\rho(\underline{S})$ is order reversing and that $\mu_\rho$ satisfies the two properties required in the definition of a measure. This
is explicitly done in \cite{ms}.

On the other hand, an abstract measure on the spectral presheaf associated to any given algebra $N$, determines a unique state of $N$, provided $N$ contains no direct summand of type$I_2$. The proof of this rather surprising result uses a generalized version of Gleason's theorem, which can be found in \cite{26}. 

\begin{definition}
A finitely additive probability measure $m$ on the projections of a von Neumann algebra $N$ is a map
$$m:\mathcal{P}(N)\rightarrow [0,1]$$
such that  $m(I)=1$ and if $P$ and $Q$ are orthogonal projections then $$m(P\vee Q)=m(P+Q)=m(P)+m(Q)$$
\end{definition}

\begin{theorem}[Gleason]\label{gleason}
Each finitely additive probability measure on the projections of a von Neumann algebra without type $I_2$ summands, can be uniquely extended to a state on that algebra.
\end{theorem}

Using this powerful result we can show that each measure on the spectral presheaf uniquely determines a state on the corresponding algebra by showing that such a measure determines a unique finitely additive probability measure on the projections of the respective algebra. This has been done by D\"{o}ring in \cite{ms}, and we will reproduce his proof in the remainder of this section.

Given a measure $\mu$ on the spectral presheaf $\underline{\Sigma}$ associated to some von Neumann algebra $N$, let $\underline{S}$ be a clopen subobject of $\underline{\Sigma}$. From Proposition \ref{alpha} we know that for each context $V$ there exists an isomorphism $\alpha_V$ between $\mathcal{P}(V)$ and $Cl(\underline{\Sigma}_V)$. If $P=\alpha_V^{-1}(\underline{S}_V)$ we define 
$$m(P)=\mu(\underline{S})(V)=\mu(\underline{S}_V)$$
We have to show that this does not depend on the choice of the subobject $\underline{S}$ and the context $V$, i.e. we must show that if $\underline{\tilde{S}}$ is another subobject of $\underline{\Sigma}$ and $\tilde{V}$ is a context such that 
$\alpha_V^{-1}(\underline{S}_V)=\alpha_{\tilde{V}}^{-1}(\underline{\widetilde{S}}_{\tilde{V}})$
then $\mu(\underline{S}_V)=\mu(\underline{\tilde{S}}_{\tilde{V}})$.
For this we will need two intermediate results.

\begin{lemma}
If $\underline{S}$ is a clopen subobject of $\underline{\Sigma}$ and $V'\subseteq V$ are two contexts such that $P$ is contained in both $V$ and $V'$ and $\alpha_V^{-1}(\underline{S}_V)=\alpha_{V'}^{-1}(\underline{S}_{V'})=P$, then $\mu((\underline{S}_V)=\mu((\underline{S}_{V'})$.
\end{lemma}

\begin{proof}
Since the maximal projection $I$ is contained in every context, it follows that $I-P\in V',V$. Let $\underline{S^c}$ be another clopen subobject such that 
$$\alpha_V^{-1}(\underline{S^c}_V)=\alpha_{V'}^{-1}(\underline{S^c}_{V'})=I-P$$
Such a subobject certainly exists: $\underline{\delta(I-P)}$, for example, satisfies the above property.

Since every $\alpha$ is a lattice isomorphism, we have
\begin{align*}
(\underline{S}\wedge \underline{S^c})_V=0_V=\emptyset \ \  &, \ \ \ \ (\underline{S}\wedge \underline{S^c})_{V'}=0_{V'}=\emptyset \\
(\underline{S}\vee\underline{S^c})_V=\underline{\Sigma}_V \ \ &, \ \ \ \ (\underline{S}\vee\underline{S^c})_{V'}=\underline{\Sigma}_{V'}
\end{align*}

Using the two defining properties of a measure $\mu$ we obtain
\begin{align*}
1&=\mu(\underline{\Sigma})(V)\\
&=\mu(\underline{S}\vee\underline{S^c})(V)\\
&=\mu(\underline{S})(V)+\mu(\underline{S^c})(V)-\mu(\underline{S}\wedge\underline{S^c})(V)
\end{align*}
Since the last term vanishes we obtain that $\mu(\underline{S})(V)+\mu(\underline{S^c})(V)=1$. Similarly, we can also deduce that $\mu(\underline{S})(V')+\mu(\underline{S^c}(V')=1$. But $\mu(S):\mathcal{V}(N)\rightarrow [0,1]$ is an order-reversing function, hence
\begin{align*}
&\mu(\underline{S})(V')\geq \mu(\underline{S})(V)\\
&\mu(\underline{S^c})(V')\geq \mu(\underline{S^c})(V)
\end{align*}
This implies that in fact $\mu(\underline{S})(V')=\mu(\underline{S})(V)$ and $\mu(\underline{S^c})(V')=\mu(\underline{S^c})(V)$, which completes our proof. \qed
\end{proof}

\begin{lemma}
If $\underline{S}$ and $\underline{\tilde{S}}$ are two subobjects which coincide at $V$, i.e. if $\underline{S}_V=\underline{\tilde{S}}_V$, then $\mu(\underline{S})(V)=\mu(\underline{\tilde{S}})(V)$.
\end{lemma}

\begin{proof}
From the second defining property of a measure $\mu$ we obtain that
\begin{align*}
\mu(\underline{S})(V)+\mu(\underline{\tilde{S}})(V)&=\mu(\underline{S}\vee\underline{\tilde{S}})(V)+\mu(\underline{S}\wedge\underline{\tilde{S}})(V)\\
&=\mu((\underline{S}\vee\underline{\tilde{S}})_V)+\mu((\underline{S}\wedge\underline{\tilde{S}})_V)\\
&=\mu(\underline{S}_V\cup\underline{\tilde{S}}_V)+\mu(\underline{S}_V\cap\underline{\tilde{S}}_V)\\
&=\mu(\underline{S}_V)+\mu(\underline{S}_V)\\
&=\mu(\underline{S})(V)+\mu(\underline{S})(V)
\end{align*}

Which implies that $\mu(\underline{S})(V)=\mu(\underline{\tilde{S}})(V)$.\qed
\end{proof}

Now assume that $\underline{S}$ and $\underline{\tilde{S}}$ are two clopen subobjects of $\underline{\Sigma}$ and $V$ and $\tilde{V}$ are two contexts such that $\underline{S}_V$ and $\underline{\tilde{S}}_{\tilde{V}}$ correspond to the same projection $P\in V,\tilde{V}$. Then we must have that $P$ also belongs to $V\cap\tilde{V}$. We know that the clopen subobject $\underline{\delta(P)}$ coincides with $\underline{S}$ at $V$ and it also coincides with $\underline{\tilde{S}}$ at $\tilde{V}$. Moreover, $\underline{\delta(P)}_{V\cap\tilde{V}}\subseteq \underline{\Sigma}_{V\cap\tilde{V}}$ and  $\alpha_{V\cap\tilde{V}}^{-1}(\underline{\delta(P)}_{V\cap\tilde{V}})=P$. From the previous two lemmas we obtain that

\begin{align*}
\mu(\underline{S})(V)&=\mu(\underline{\delta(P)})(V)\\
&=\mu(\underline{\delta(P)})(V\cap\tilde{V})\\
&=\mu(\underline{\delta(P)})(\tilde{V})\\
&=\mu(\underline{\tilde{S}})(\tilde{V})
\end{align*}

This shows that the value $m(P)=\mu(\underline{S})(V)$ is well defined. For any $V$, the projection corresponding to $\underline{\Sigma}_V$ is the maximal projection, $I$. So from the first defining property of a measure $\mu$, we must have $$m(I)=\mu(\underline{\Sigma})(V)=1$$ 
Finally, let $P$ and $Q$ be two orthogonal projections and let $V$ be a context that contains both $P$ and $Q$. Let $\underline{S^P}$ and $\underline{S^Q}$ be two subobjects such that $\alpha_V^{-1}(\underline{S^P}_V)=P$ and $\alpha_V^{-1}(\underline{S^Q}_V)=Q$. Then $(\underline{S^P}\vee\underline{S^Q})_V$ corresponds to $P\vee Q$ and we obtain
\begin{align}
m(P\vee Q)&=\mu(\underline{S^P}\vee \underline{S^Q})(V)\\
&=\mu(\underline{S^P})(V)+\mu(\underline{S^Q})(V)+\mu(\underline{S^P}\wedge\underline{S^Q})(V)\\
&=\mu(\underline{S^P})(V)+\mu(\underline{S^Q})(V)\\
&=m(P)+m(Q)
\end{align}

This shows that the map $m:\mathcal{P}\rightarrow [0,1]$ is indeed a finitely additive probability measure, and so from the generalised version of Gleason's theorem we know that $m$ extends to a unique state $\rho_m$ of the algebra $N$.

In particular this implies that when the algebra $N$ is a finite dimensional matrix algebra there is a bijective correspondence between density matrices and measures on the corresponding spectral presheaf.

\section{Contextual entropy}\label{ctxt}

\subsection{Measures and partial traces}

We saw that, given a measure $\mu$ on the clopen subobjects of a spectral presheaf, if we fix a subobject $\underline{S}$ of $\underline{\Sigma}$ we obtain a map from
$\mathcal{V}(N)$ to $[0,1]$. We can adopt a different perspective and instead of looking at a fixed subobject we can look at a fixed context $V$. There is a lattice isomorphism $\alpha_V$ between the projections in $V$ and the clopen subsets of $\underline{\Sigma}_V$. Hence from $\mu$ we can also obtain a map 
\begin{align*}
\mu|_{_V}:\mathcal{P}(V)&\longrightarrow[0,1]\\
 P&\longmapsto \mu(S_P)
\end{align*}
where $S_P=\alpha_V(P)\subseteq \underline{\Sigma}_V$.

Using this new perspective, we can show that measures on the spectral presheaf associated to a matrix algebra behave well with respect to the partial trace. This result has a certain physical significance. We have already seen that there is a bijective correspondence between states and probability measures, and we now show that moreover these measures capture the essential information theoretic property of the partial trace in a natural way. Thus, if we are given a measure corresponding to a composite state, we can obtain its partial traces in a direct way by simply considering its restrictions to contexts of a particular form. Intuitively, we would expect these contexts to be precisely those which only encode information related to the first subsystem (if we want to trace out the second one) or vice versa, and we will see that this will indeed be the case.

Note also that this result will be useful for us later on, when discussing the subadditivity property of our contextual entropy.

\begin{proposition}
Consider a state $\rho$ on the matrix algebra $\mathcal{M}_{nm}= \mathcal{M}_n\otimes \mathcal{M}_m$. Let $\rho_1=\mathrm{Tr}_2(\rho)\in\mathcal{M}_n$ and $\rho_2=\mathrm{Tr}_1(\rho)\in\mathcal{M}_m$ be the partial traces of $\rho$. Then if $V\in \mathcal{V(M}_n)$ and $\mathbb{C}I_m$ denotes the trivial subalgebra of $\mathcal{M}_m$ we have
$$\mu_\rho|_{_{V\otimes \mathbb{C}I_m}}=\mu_{\rho_1}|_{_V}$$
Conversely, if $W\in \mathcal{V(M}_m)$ and $\mathbb{C}I_n$ denotes the trivial subalgebra of $\mathcal{M}_n$ we have
$$\mu_\rho|_{_{\mathbb{C}I_n\otimes W}}=\mu_{\rho_2}|_{_W}$$
\end{proposition}

\begin{proof}
To see that this is indeed the case note first that there is a lattice isomorphism between the domains of definition of $\mu_\rho|_{_{V\otimes \mathbb{C}I_m}}$ and $\mu_{\rho_1}|_{_V}$ which takes $P\in\mathcal{P}(V)$ to $P\otimes I_m\in\mathcal{P}(V\otimes \mathbb{C}I_m)$. Then using the definition of measures for states on matrix algebras and the defining property of the partial trace, we have that 
$$\mu_\rho|_{_{V\otimes \mathbb{C}I_m}}(P\otimes I_m)=Tr(\rho\cdot P\otimes I_m)=Tr(\rho_1\cdot P)=\mu_{\rho_1}|_{_V}(P),\ \ \forall P\in\mathcal{P}(V\otimes \mathbb{C}I_m)$$
and similarly for the second statement.\qed 
\end{proof}

Finally, the fact that $\mu$ is a measure implies several properties for $\mu|_{_V}$ which hold for all contexts $V\in\mathcal{V}(N)$, and which we shall state below:
\begin{enumerate}
 \item $\mu|_{_V}(I)=1$ and $\mu|_{_V}(0)=0$
 \item $\mu|_{_V}(P\vee Q)+\mu|_{_V}(P\wedge Q)=\mu|_{_V}(P)+\mu|_{_V}(Q)$
 \item in particular, if $P$ and $Q$ are orthogonal then $P\wedge Q=0$ and $P\vee Q=P+Q$ and hence $$\mu|_{_V}(P+Q)=\mu|_{_V}(P)+\mu|_{_V}(Q)$$
 \item if $P\leq Q$ then $\mu|_{_V}(P)\leq \mu|_{_V}(Q)$
\end{enumerate}

These properties imply that $\mu|_{_V}$ is a finitely additive probability measure on the lattice of projections of $V$.

\subsection{The entropy of a measure}

We saw that in classical probability theory we can define Shannon entropy as a function on the set of all probability distributions. We will see now how to associate a distinguished probability distribution to each context of a von Neumann algebra of bounded operators on finite dimensional Hilbert space, given a state on the system described by that algebra in the form of a measure on its associated spectral presheaf. Once this is done, we will be able to associate to each context its corresponding Shannon entropy, and moreover we will see that this collection of Shannon entropies fits together in a nice way and gives a global section of a certain real-number presheaf. This is consistent with the basic idea of the topos approach, that of putting together the information obtained from each classical perspective on a quantum system. We will see in later sections that by keeping track of all classical entropies associated to a quantum state we can not only retrieve that state's von Neumann entropy, but also reconstruct the state itself.

\begin{definition}
Let $H$ be an Hilbert Space, $\mathcal{B}(H)$ the algebra of bounded operators in $H$ and $\mathcal{F}\subseteq \mathcal{B}(H)$. The von Neumann commutant of $\mathcal{F}$, usually denoted by $\mathcal{F}'$, is the subset of $\mathcal{B}(H)$ consisting of all elements that commute with every element of $\mathcal{F}$, that is
$$\mathcal{F}'=\{T\in \mathcal{B}(H)~|~TS=ST,\  \forall S\in\mathcal{F}\}$$

The von Neumann double commutant $\mathcal{F}$ of is just $(\mathcal{F}')'$ and is usually denoted by $\mathcal{F}''$.
\end{definition}

If we consider a set of orthogonal rank-one projections $\{P_1,\ldots,P_n\}''$, their double commutant can be shown to be simply $\mathbb{C}P_1+\ldots+\mathbb{C}P_n$.

It is known that each context $V$ can be generated via the von Neumann double commutant construction in a unique way from a set of pairwise orthogonal projections which add up to the identity. If we denote this canonical set of projections by $\{P_1, P_2,\ldots,P_k\}$ then $(\mu|_{_V}(P_1), \mu|_{_V}(P_2),\ldots,\mu|_{_V}(P_k))$ is a probability distribution. Hence to each context $V$ we can assign the Shannon entropy of its associated probability distribution:
$$\mathrm{Sh}(\mu|_{_V}(P_1), \mu|_{_V}(P_2),\ldots,\mu|_{_V}(P_k))=-\sum_{i=1}^k \mu|_{_V}(P_i)\ln \mu|_{_V}(P_i)$$
If $V'\supseteq V$ then $V' = \{Q^1_1,\ldots,Q^1_{l_1}, Q^2_1,,\ldots Q^2_{l_2},\ \ldots,\ Q^k_1,\ldots,Q^k_{l_k}\}''$, where the $Q^j_i$s are pairwise orthogonal and
$$\sum_{i=1}^{k_j} Q^j_i=P_j$$

The Shannon entropy associated to $V'$ is related to the Shannon entropy associated to $V$ via the recursion formula:
$$\mathrm{Sh}(V')=\mathrm{Sh}(V)+\sum_{i=1}^k \mu|_{_V}(P_i)\cdot \mathrm{Sh}\left(\frac{\mu|_{_{V'}}(Q^i_1)}{\mu|_{_V}(P_i)}, \frac{\mu|_{_{V'}}(Q^i_2)}{\mu|_{_V}(P_i)},\ldots,
\frac{\mu|_{_{V'}}(Q^i_{l_i})}{\mu|_{_V}(P_i)}\right)$$

Since Shannon entropy is non-negative, it follows that $\mathrm{Sh}(V')\geq\mathrm{Sh}(V)$ and this enables us to give the following definition for the entropy of a measure (and hence of a quantum state).

\begin{definition}
If $\mu$ is a measure on the clopen subobjects of a presheaf $\underline{\Sigma}$ then the entropy $E(\mu)$ associated to $\mu$ is a global section of the presheaf $\underline{[0,\ln n]^\preceq}$ which at a context $V=\{P_1, P_2,\ldots,P_k\}''$ has the value $$E(\mu)|_{_V}=\mathrm{Sh}(\mu|_{_V}(P_1),
\mu|_{_V}(P_2),\ldots,\mu|_{_V}(P_k))=-\sum_{i=1}^k \mu|_{_V}(P_i)\ln \mu|_{_V}(P_i)$$
Note that if the $V$ is a $k$-dimensional context then the value taken by $E(\mu)$ at $V$ is less then or equal to $\ln k$, and hence for an $n$-dimensional matrix algebra, the maximal value taken by $E(\mu)$ at any context is $\ln n$. Therefore contextual entropy can be seen as a mapping defined on the set of measures associated to a spectral presheaf:
$$E:\mathcal{M}(\underline{\Sigma})\longrightarrow \Gamma \underline{[0,\ln n]^{\preceq}}\ \  .$$
\end{definition}

Notice that although there is a bijective correspondence between states of a von Neumann algebra and measures on the spectral presheaf associated to it, the above definition does
not make any direct reference to the quantum state which the measure corresponds to.

\subsection{Properties of the contextual entropy}

\subsubsection{Extracting the von Neumann entropy}\label{VNeu}
Given a density matrix $\rho$ there exists at least one orthonormal basis of Hilbert space with respect to which $\rho$ is diagonal. Such a basis corresponds to a set of
one-dimensional pairwise orthogonal projections $\{P_1,\ldots,P_n\}$, which in turn determine a maximal context $V_\rho$ via the double commutant construction. It is easy to check
that the eigenvalues $\{\lambda_i\}_{i=1}^n$ of $\rho$ satisfy $\lambda_i=Tr(\rho P_i)$. Hence the value assigned to the entropy of the measure $\mu_\rho$ at any context $V_\rho$
obtained through the above procedure, is just the von Neumann entropy of the state $\rho$:
\begin{align*}
E(\mu_\rho)_{_{V_\rho}} &=-\sum_{i=1}^n \mu_\rho|_{_{V_\rho}}(P_i)\ln \mu_\rho|_{_{V_\rho}}(P_i)\\ 
&=-\sum_{i=1}^n \mathrm{Tr}(\rho P_i)\ln\mathrm{Tr}(\rho P_i)\\
&=-\sum_{i=1}^n \lambda_i\ln\lambda_i=\mathrm{VN}(\rho)
\end{align*}

One can prove (see for example Wehner and Short \cite{stephanie}, Appendix B) that for any other maximal context $V$, the associated Shannon entropy is strictly larger than the Shannon entropy associated with $V_\rho$. This is a consequence of the so-called Schur-Horn Theorem \cite{horn} and the Schur-concavity of Shannon entropy. Thus, the von Neumann entropy of $\rho$ is equal to the minimal value of the contextual entropy $E_{\rho}(V)$, when $V$ is varying over the set of maximal contexts.

Given the contextual entropy map, the problem of finding a context for which this minimum is attained is equivalent to the problem of finding the point at which a real-valued function on the group of unitaries $\mathcal{U}(n)$ attains its minimal value. To see why this is the case, let $V_0:=\{{E}_1,\ldots,{E}_n\}''$ denote the maximal context determined by projections which are diagonal with respect to the computational basis. Any other maximal context $V=\{P_1,\ldots, P_n\}''$ can be written as $U.V_0:=\{U{E}_1U^{-1},\ldots,U{E}_nU^{-1}\}''$ for some unitary $U$. Hence if we restrict the contextual entropy map to the set $\mathcal{V}_M$ of maximal contexts, we can write $E_\rho|_{\mathcal{V}_M}(C)=E_\rho|_{\mathcal{V}_M}(U.C_0)$ and we can view this restriction as a real-valued function on the group of unitaries. It is then possible to use existing optimization algorithms \cite{traian1, traian2} in order to find, with high probability of success, the point at which this function attains its global minimum.

\subsubsection{Unitarily equivalent global sections} \label{unitarily equivalent global sections}

If we evaluate the contextual entropy of a state $\rho$ at some  context $V$ (not necessarily maximal), this will be equal to the contextual entropy of any unitarily equivalent state as long as we evaluate it at a context which is obtained from $V$ through rotation by the same unitary. That is,
$$E(\mu_\rho)_{_V}=E(\mu_{U\rho U^{-1}})_{_{U\cdot V\cdot U^{-1}}}$$

This observation has a certain physical significance. Within the Schr\"{o}dinger approach to quantum mechanics, one uses unitary transformations of a state in order to encode time evolution of that state. On the other hand, one can use Heisenberg approach to encode time evolution, and then one looks at unitary transformations of the coordinate systems in which the states are represented. Of course, the laws of physics should not depend on which interpretation of quantum mechanics we choose to follow, and this is exactly what the above equation captures.

\subsubsection{Contextual vs. Von Neumann and Shannon entropies}\label{cns}

We would like at this point to compare the properties of Von Neumann and Shannon entropies with those of the contextual entropy. The main difficulty with this attempt is the fact
that the values of the contextual entropies are not real numbers but global sections of certain real number presheaves, which may live in different topoi, i.e. they may be defined over different base categories. In some cases it is possible to work around this difficulty by adapting the definitions of order relations and algebraic operations on $\mathbb{R}$ to suit our more general framework. 

\vspace{12pt}
\noindent \textbf{1) Positivity}

Both von Neumann and Shannon entropies are positive. Shannon entropy is zero for any probability distribution in which one outcome occurs with $100$\% certainty and strictly
positive otherwise. Similarly, von Neumann entropy is zero for all pure states, and strictly positive for the others.

The contextual entropy does assign non-negative values to all contexts, hence the resulting global section can be thought of as non-negative, but it does not assign the value zero to all contexts for pure states. However, we can still recognize pure states because, as we have already seen, it is possible to determine the Von Neumann entropy from the contextual one by taking the minimum over all values assigned to maximal contexts.

The advantage of using this richer notion of entropy is that not only can we distinguish pure states from non-pure ones, but by considering all contexts at the same time we encode
sufficient information to reconstruct the pure state itself. Moreover it is possible, with a few exceptions, to reconstruct any quantum state from our contextual entropy, and we
shall see how this is done later on.

\noindent \textbf{2) Concavity}

Shannon entropy is concave: if $\vec{p}$ and $\vec{q}$ are two probability distributions then $$\mathrm{Sh}(r\cdot \vec{p}+(1-r)\cdot \vec{q})\geq
r\mathrm{Sh}(\vec{p})+(1-r)\mathrm{Sh}(\vec{q})$$ For von Neumann entropy concavity is defined by a similar formula: $$\mathrm{VN}(r\rho+(1-r)\sigma)\geq
r\mathrm{VN}(\rho)+(1-r)\mathrm{VN}(\sigma)$$

The contextual entropy satisfies a similar property. If $\rho$ and $\sigma$ are defined on the same Hilbert space $\mathcal{H}$ then for every context $V\in\mathcal{B(H)}$, if $V$
is generated by the projections $\{P_1,\ldots,P_k\}$, we have
\begin{align*}
E(\mu_{r\rho+(1-r)\sigma})_{_V}&=\mathrm{Sh}(~ \mathrm{Tr}[(r\rho+(1-r)\sigma)P_1],\, \ldots,\mathrm{Tr}[(r\rho+(1-r)\sigma)P_k]~ )\\
 &=\mathrm{Sh}(~ [r\mathrm{Tr}(\rho P_1)+(1-r)\mathrm{Tr}(\sigma P_1)],\, \ldots,r\mathrm{Tr}(\rho P_k)+(1-r)\mathrm{Tr}(\sigma P_k)~)\\
 &\geq r\mathrm{Sh}(\mathrm{Tr}(\rho P_1),\, \ldots,\mathrm{Tr}(\rho P_k))~+~(1-r)\mathrm{Sh}(\mathrm{Tr}(\sigma P_1),\, \ldots,\mathrm{Tr}(\sigma P_k))\\
 &=r\cdot E(\mu_\rho)_{_V}+(1-r)E(\mu_\sigma)_{_V}
\end{align*}

Hence we are justified to say that contextual entropy is globally concave:
$$E(\mu_{r\rho+(1-r)\sigma})\geq r\cdot E(\mu_\rho)+(1-r)E(\mu_\sigma), \ \ \forall r\in[0,1]$$

\noindent \textbf{3) Additivity and Subadditivity}

Subadditivity a property concerning composite systems. Recall that an entropy is called subadditive if the entropy of a composite system is smaller than the sum of the entropies of its parts. Both von Neumann and Shannon entropies are subadditive. We would like to obtain an inequality of the form
$$E(\mu_\rho)\leq E(\mu_{\rho_1})+E(\mu_{\rho_2})$$
where $\rho$ is the density matrix representing a composite state and $\rho_1$ and $\rho_2$ are the partial traces of $\rho$. It is not immediately clear how one could define such an inequality, since this time the terms involved are global sections of presheaves over three different base categories.
Hence in order to talk about subadditivity in a meaningful way, we must first define a suitable notion of addition between the global sections $E(\mu_{\rho_1})$ and
$E(\mu_{\rho_2})$.

In order to see how this might be done, we start by considering some context $V$ of the first subsystem and some other context $W$ of the second subsystem. If
$V=\{P_1,\ldots,P_k\}''$ and $W=\{Q_1,\ldots,Q_r\}''$, from the definition of the entropy we have
\vspace{-8pt}
$$ E(\mu_{\rho_1})|_{_V}=\sum_{i=1}^k \mathrm{Tr}(\rho_1 P_i) \ln  \mathrm{Tr}(\rho_1 P_i),$$
\vspace{-10pt}
$$ E(\mu_{\rho_2})|_{_W}=\sum_{j=1}^r \mathrm{Tr}(\rho_2 Q_j) \ln  \mathrm{Tr}(\rho_2 Q_j)$$
We can add these two numbers together, and we can use the fact that Shannon entropy is additive for independent probability distributions (i.e. $\sum_{i=1}^k p_i\ln p_i+
\sum_{j=1}^r q_j \ln q_j = \sum_{i,j} p_iq_j \ln p_iq_j$) and the fact that $ \mathrm{Tr}(\rho_1 P_i) \mathrm{Tr}(\rho_2 Q_j)= \mathrm{Tr}(\rho_1\otimes \rho_2 P_i\otimes Q_j)$ to
obtain
$$E(\mu_{\rho_1})|_{_V}+E(\mu_{\rho_2})|_{_W} = \sum_{i=1,j} \mathrm{Tr}(\rho_1\otimes \rho_2 P_i\otimes Q_j) \ln \mathrm{Tr}(\rho_1\otimes \rho_2 P_i\otimes Q_j) =
E(\mu_{\rho_1\otimes\rho_2})|_{_{V\otimes W}}$$

It makes sense then to use the following requirement for the definition of subadditivity: $E(\mu_\rho)$ should be less than or equal to $E(\mu_{\rho_1\otimes\rho_2})$ at each
context $\widetilde{V}$ of the composite system. This definition enables us to say, for instance, that the contextual entropy is additive when $\rho=\rho_1\otimes\rho_2$. Note that this is a direct consequence of the additivity property of Shannon entropy.

Even when $\rho$ is not equal to $\rho_1\otimes\rho_2$ the subadditivity property holds in split contexts (i.e. contexts of the form $V\otimes W$) as a consequence of Shannon subadditivity. Consider $\widetilde{V}=V\otimes W$, with $V$ and $W$ as above. We know that 
$$\mu_{\rho_1}(P_i) = \mu_\rho(P_i\otimes I) = \sum_{j=1}^r \mu_\rho(P_i \otimes Q_j)$$
for all $i\in\{1,\ldots,k\}$ and 
$$\mu_{\rho_2}(Q_j) = \mu_\rho(I\otimes Q_j)= \sum_{j=1}^r \mu_\rho(P_i \otimes Q_j)$$
for all $j\in\{1,\ldots,r\}$. 

Using the subadditivity property of Shannon entropy we obtain 
\begin{align*}
 E(\mu_\rho)|_{_{V\otimes W}} &= \mathrm{Sh}(P_1\otimes Q_1,\ldots,P_1\otimes Q_r, P_2\otimes Q_1,\ldots,P_2\otimes Q_r,\ldots,P_k\otimes Q_1,\ldots, P_k\otimes Q_r) \\
&\leq \mathrm{Sh}\left(\sum_{i=1}^k \mu_\rho(P_i\otimes Q_1), \sum_{i=1}^k \mu_\rho(P_i\otimes Q_2),\ldots, \sum_{i=1}^k \mu_\rho(P_i\otimes Q_r)\right) + \\
& \ \ ~ \ \mathrm{Sh}\left(\sum_{j=1}^r \mu_\rho(P_1\otimes Q_j), \sum_{j=1}^r \mu_\rho(P_2\otimes Q_j),\ldots, \sum_{j=1}^r \mu_\rho(P_k\otimes Q_j)\right)\\
&=\mathrm{Sh}\left(\mu_{\rho_1}(P_1),\mu_{\rho_1}(P_2),\ldots, \mu_{\rho_1}(P_k)\right)+\mathrm{Sh}(\mu_{\rho_2}(Q_1),\mu_{\rho_2}(Q_2),\ldots,\mu_{\rho_2}(Q_r))\\
&=E(\mu_{\rho_1})|_{_V}+E(\mu_{\rho_2})|_W=E(\mu_{\rho_1\otimes \rho_2})|_{_{V\otimes W}} 
\end{align*}

\begin{remark}
The fact that the contextual entropy is subadditive in all split contexts can be used to give a more direct proof of the subadditivity property of von Neumann entropy, which avoids using Klein's inequality: if we choose the split context $\widetilde{V}$ such that $\rho_1$ is diagonal in $W$ and $\rho_2$ is diagonal in $W$ we know from Section \ref{VNeu} that
$$\mathrm{VN}(\rho)\leq E(\mu_\rho)|_{_{V\otimes W}} \leq E(\mu_{\rho_1})|_{_V}+E(\mu_{\rho_2})|_W = \mathrm{VN}(\rho_1)+\mathrm{VN}(\rho_2)$$
\end{remark}

For contexts which are not split (which we usually call entangled contexts), the subadditivity property does not hold in general, and one can construct explicit counterexamples. This is not surprising, since the converse of the Schur-Horn lemma implies that for
any density matrix $\rho$, there is some unitary $U$ for which the diagonal of $U\rho U^{-1}$ is the maximally mixed vector $(\frac{1}{n},\frac{1}{n},\ldots,\frac{1}{n})$. Let $D_n$ denote the context generated by the set of projections $\{E_{11},\ldots,E_{nn}\}$, where we have fixed our basis such that $E_{ii}$ is the projection with the $i^{th}$ diagonal entry equal to one and all other entries equal to zero. Then

$$E(\mu_\rho)|_{_{U^{-1}\cdot D_n\cdot U}}=E(\mu_{U\rho U^{-1}})|_{_ {D_n}}=\mathrm{Sh}\left((U\rho U^{-1})_{11},\ldots,(U\rho U^{-1})_{nn}\right)=\ln n$$

There is however no guarantee that the diagonal of $U\rho_1\otimes\rho_2U^{-1}$ will also be the maximally mixed vector.

\noindent \textbf{4) Continuity}

It is possible to define a metric on each set of unitarily equivalent contexts. One can show that the contextual entropy, seen as a real-valued function on such a set of unitarily equivalent contexts, is continuous with respect to this metric.

Since the projections which generate a given context are unique up to permutations and multiplication by phase factors, these operations should not influence the distance between two contexts. If $V=\{P_1,\ldots,P_k\}''$ and $W=\{Q_1,\ldots,Q_k\}''$ are two $k$-dimensional contexts, we say that $V$ is unitarily equivalent to $W$ if there exists a unitary $U$ such that $UP_iU^{-1}=Q_{\sigma(i)}$ for some permutation $\sigma\in S_n$. Let $\mathcal{U}_{V,W}$ be the collection of all such unitaries. Each element in $\mathcal{U}_{V,W}$ is a representative of an equivalence class of unitaries in $ U(n)/U(1)^n$. 

We define the distance between two unitarily equivalent contexts as 
$$d(V,W)=\min_{\substack{U\in \mathcal{U}_{V,W}\subseteq U(n)\\ \tilde U \in [U]\in U(n)/U(1)^n}}\ ||\tilde U-I||$$
where $||T||=\max_{i,j}|T_{ij}|$. This is clearly well defined and satisfies the conditions required for a metric. 

Given a state $\rho$ and a set of unitarily equivalent contexts $\mathcal{V}$, the contextual entropy of the state $\rho$ can be seen as a function on this set:
\begin{align*}
E_\rho:\mathcal{V}&\rightarrow [0,\infty)\\
V&\mapsto E(\mu_\rho)|_{_V}
\end{align*}

We can check that this function is continuous with respect to the previously defined metric. Let $\epsilon$ be a matrix such that $U:=I+\epsilon$ is a unitary. For any context $V=\{P_1,\ldots,P_k\}''$ we have
\begin{align*}
E_\rho(UVU^{-1})&-E_\rho(V)=\sum_{i=1}^k Tr(U^{-1}\rho UP_i)\ln Tr(U^{-1}\rho UP_i)-\sum_{i=1}^kTr(\rho P_i)\ln Tr(\rho P_i)\\
&=\sum_{i=1}^kTr((I+\epsilon^*)\rho(I+\epsilon)P_i)\ln Tr((I+\epsilon^*)\rho(I+\epsilon)P_i) - \sum_{i=1}^k Tr(\rho P_i)\ln Tr(\rho P_i)\\
&=\sum_{i=1}^k [Tr(\rho P_i)+\underbrace{Tr(\epsilon^*\rho P_i) + Tr(\rho \epsilon P_i) + Tr(\epsilon^*\rho\epsilon P_i)}_{A}]\ln [Tr(\rho P_i)+\\
&\ \ \ \ \ \ \  +Tr(\epsilon^*\rho P_i) + Tr(\rho \epsilon P_i) + Tr(\epsilon^*\rho\epsilon P_i)] -  \sum_{i=1}^kTr(\rho P_i)\ln Tr(\rho P_i)\\
&=\sum_{i=1}^k [Tr(\rho P_i)+ A][\ln Tr(\rho P_i)+\frac{1}{Tr(\rho P_i)}A+\mathcal{O}(A^2)]- \sum_{i=1}^kTr(\rho P_i)\ln Tr(\rho P_i)\\
&=\sum_{i=1}^k A[\ln Tr(\rho P_i)+1]+ \mathcal{O}(A^2)\\
&=\sum_{i=1}^k [Tr(\epsilon^*\rho P_i) + Tr(\rho \epsilon P_i) + Tr(\epsilon^*\rho\epsilon P_i)][\ln Tr(\rho P_i)+1] + \mathcal{O}(A^2)\\
&\leq C||\epsilon|| Tr(\rho P_i)[\ln Tr(\rho P_i)+1] + \mathcal{O}(||\epsilon||^2)
\end{align*}
for some finite constant $C$. This shows that $E_\rho$ is indeed continuous.

\subsection{Reconstructing pure states from global sections}\label{r1}

A direct implication of Remark \ref{unitarily equivalent global sections} is that unlike von Neumann entropy, which gives the same value for unitarily equivalent states, our contextual entropy gives different (though in a sense unitarily equivalent) global sections of the presheaf $\underline{[0,\ln n]^\preceq}$. This enables us not only to distinguish which global sections come from measures associated to pure states but also to explicitly reconstruct those pure states. We explain this method in more detail.

Recall that the von Neumann entropy of a state vanishes if and only if that state is a pure one. Given a global section $\gamma\in \Gamma\underline{[0,\ln n]^\preceq}$ if $\gamma$
is in the image of the contextual entropy mapping $E$ then it comes from a measure associated to a pure state if and only if there exists a maximal context $V$ such that
$\gamma|_{_V}=0$. This means that if $V$ is generated by the set of rank one projections $$\{P_1,\ldots,P_n\}=\{\left|\psi_1\right>\left<\psi_1\right|,\ldots,
\left|\psi_n\right>\left<\psi_n\right|\}$$ our state must equal one of these projections and our only task is to determine which one. For this, consider unitaries $U_1,\ldots,U_n$
which have the property that $U_iP_iU^{-1}_i=P_i$ and $$\{U_iP_jU^{-1}_i~|~1\leq j\leq n, j\neq i\}\neq\{P_1,\ldots,\widehat{P_i},\ldots,P_n\}$$
Think of this as taking $n$ rotations in Hilbert space, each of which preserves one axis of the orthonormal basis $\{\left|\psi_1\right>,\ldots,\left|\psi_n\right>\}$ and rotates
the others, but without permuting them.

If we consider the contexts $V_i= \{U_iP_1U_i^{-1},\ldots,U_iP_nU_i^{-1}\}''$ then $\rho$ will be diagonal only in one of the orthonormal bases which correspond to these contexts.
This means the contextual entropy will assign the value zero to precisely one of the contexts $V_i$, and hence our state is
$$\rho= \{U_iP_1U_i^{-1},\ldots,U_iP_nU_i^{-1}\}\cap \{P_1,\ldots,P_n\}$$

\subsection{Reconstructing arbitrary quantum states from global sections}

Consider a global section $\gamma\in \Gamma\underline{[0,\ln n]^\preceq}$. We present here an algorithm for reconstructing the state $\rho$ for which $E(\mu_\rho)=\gamma$. We
assume for now that $\gamma$ is in the image of the contextual entropy mapping. If our algorithm will fail to find a solution we will know that our initial assumption was false.
Otherwise we must perform one final check at the end of our algorithm to make sure that this assumption was correct.

Start by identifying one maximal context $V$ such that $\gamma|_{_V}\leq\gamma|_{_W}$ for all maximal contexts $W$. This amounts to retrieving the von Neumann entropy of the state$\rho$ from the contextual entropy. If this equals zero we must have a pure state, and we already saw how to reconstruct those. Otherwise, we know from Section \ref{VNeu} that
$\rho$ must be diagonal in the context $V$. If we consider the canonical projections $\{P_1,\ldots,P_n\}$ which generate $V$, the fact that $\rho$ is diagonal at $V$ implies that
it is of the form
$$\rho=\lambda_1 P_1+\ldots+\lambda_n P_n$$
where the $\lambda_i$'s are the eigenvalues of $\rho$. We are now left with the task of determining these eigenvalues. For this assume that the dimension $n$ of our Hilbert space
is greater or equal to $3$. For each $i\in\{1,\ldots,n\}$ let $$W_i:=\{P_i,I-P_i\}''$$
Then $\textrm{Sh}(\lambda_i,1-\lambda_i)$ must equal $\gamma|_{_{W_i}}$ for all $i$. If $$\gamma|_{_{W_i}}>\ln 2$$ then the global section $\gamma$ cannot be in the image of the
contextual entropy mapping, and our algorithm stops. Otherwise, the transcendental equation $\textrm{Sh}(x_1,x_2)=k$ has two solutions which are symmetric around $\frac{1}{2}$ as
indicated in Figure \ref{Sh}.

\begin{figure}[H]
\centering
\includegraphics[width=5.2cm]{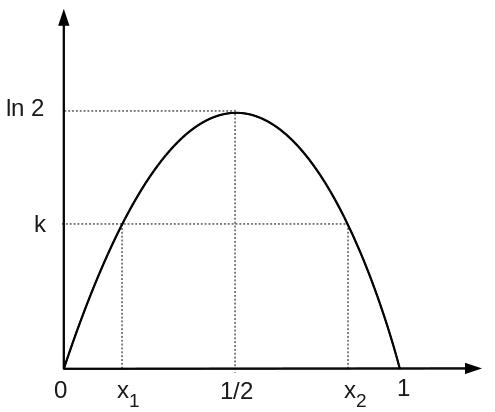}
\caption{Shannon entropy for a probability distribution with two variables}\label{Sh}
\end{figure}

\noindent Let $p_i$ and $1-p_i$ be the solutions of $\textrm{Sh}(x_1,x_2)=\gamma|_{_{W_i}}$ and assume without loss of generality that $p_i\leq \frac{1}{2}$. For each $i$ we have
at most two choices for the value of the $i^{th}$ eigenvalue of $\rho$: we can either set $\lambda_i=p_i$ or $\lambda_i=1-p_i$. Since
$$\lambda_1+\ldots+\lambda_n=1$$
there can be at most one $j$ such that $p_j < \frac{1}{2}$ and $\lambda_j=1-p_j$, while for all $i\neq j$ we must have $\lambda_i=p_i$. Let 
$$S=\sum_{i=1}^n p_i$$
Clearly $\sum_{i=1}^n\lambda_i\geq S$. We are now faced with three possible scenarios:

\begin{enumerate}
 \item If $S>1$ we obtain a contradiction, hence $\gamma$ can not be in the image of the contextual entropy mapping.
\item If $S=1$ then the assignment $\lambda_i=p_i$ gives one possible solution for the set of eigenvalues of our state $\rho$. This solution is clearly unique: any other choice of
values will make the total sum of the eigenvalues of $\rho$ greater than $1$.
 \item If $S<1$ then we must determine the $j$ for which $p_j<\frac{1}{2}$ and $\lambda_j=1-p_j$. If such a $j$ exists then
$$1=\sum_{i=1}^n\lambda_i=S-p_j+(1-p_j)$$
hence $p_j$ should equal $\frac{S}{2}$. Now
\begin{itemize}
 \item if the value $\frac{S}{2}$ does not appear amongst $\{p_1,\ldots,p_n\}$ then we have no solution 
 \item if $\frac{S}{2}$ appears once, we have a unique solution
\item if it appears more than once, let $\{j_1,\ldots,j_m\}$ be the set of indices for which $p_{j_k}=\frac{S}{2}$. If we set $\lambda_{j_k}=1-p_{j_k}$ and take another
$l\in\{1,\ldots,m\}$, $l\neq k$. Then
$$\sum_{i=1}^n\lambda_i\geq \lambda_{j_k}+ \lambda_{j_l} = 1-p_{j_k}+p_{j_k}=1$$
In order to have equality we must have $m=2$ and $p_i=0$ for all $i\notin \{j_1,j_2\}$. Unless this happens we cannot find a solution. On the other hand, for $m=2$ we have two
possible solutions. These correspond to the two states
$$\rho_1=p_{j_1} P_{j_1}+ (1-p_{j_1})P_{j_2}$$ and $$\rho_2=(1-p_{j_1}) P_{j_1}+ p_{j_1} P_{j_2}$$

In order to distinguish these two states we need to run our algorithm again but with a slight modification: instead of considering two-dimensional subalgebras of $V$, we take a
unitary $U$ which rotates all the canonical projections generating $V$, except $P_{j_1}$, which it leaves unchanged, and we consider the two dimensional subalgebras of $U\cdot
V\cdot U^{-1}$ of the form
$$\widetilde{W_i}=\{UP_iU^{-1}, I- UP_iU^{-1}\}''$$
We solve the equations $\mathrm{Sh}(x_i,1-x_i)=\gamma|_{_{\widetilde{W_i}}}$ and choose as before $n$ numbers from these solutions, such that they add up to one. These numbers
represent the diagonal entries of the matrix $U^{-1}\rho U$. We will not encounter any problems when retrieving these entries (unless of course, our initial assumption about
$\gamma$ being in the image of the contextual entropy mapping was false) because unlike the eigenvalues of $\rho$, these diagonal entries must contain more than three non-zero
elements. Moreover, the $j_1^{th}$ entry on the diagonal of $U^{-1}\rho U$ will be the same as the $j_1^{th}$ eigenvalue of $\rho$, and this tells us whether $\rho$ equals $\rho_1$or $\rho_2$.
\end{itemize}
\end{enumerate}

We have now reached the end of our algorithm. If it has failed to retrieve a solution, we conclude that we have considered a global section $\gamma$ which was not in the image of
the contextual entropy mapping. Otherwise, our reconstructed state is
$$\rho=\lambda_1 P_1+\ldots+\lambda_n P_n$$
In order to obtain $\rho$ we have taken into account only a finite number of contexts, and it might happen that when all contexts are taken into account $E(\mu_\rho)\neq\gamma$. In this case we also conclude that $\gamma$ was not in the image of the contextual entropy mapping, and discard the state $\rho$.

\subsection{Two-dimensional Hilbert spaces}\label{r3}

For two dimensional Hilbert spaces the contextual entropy is a two-to-one mapping. We will justify this statement below.

First, it is easy to check that for any one dimensional projection $P$ the states $\rho_1=\lambda P+(1-\lambda)(I-P)$ and $\rho_2=(1-\lambda)P+\lambda (I-P)$ are mapped to the same global section of $\underline{[0,\ln 2]}^\preceq$: note that $\rho_1=I-\rho_2$. Hence for every context $W=\{Q,I-Q\}''$
$$E(\mu_{\rho_1})|_{_W}=\mathrm{Sh}(\mathrm{Tr}\rho_1 Q, 1-\mathrm{Tr}\rho_1 Q)$$
while
\begin{align*}
E(\mu_{I-\rho_1})|_{_W}&=\mathrm{Sh}(\ \mathrm{Tr}(I-\rho_1) Q , \  1-\mathrm{Tr}(I-\rho_1) Q)\\
&=\mathrm{Sh}(\ \mathrm{Tr}(I-\rho_1)(I- Q), \ 1-\mathrm{Tr}(I-\rho_1)(I-Q)) 
\end{align*}
And since every one dimensional projection $Q$ has trace equal to unity,
$$\mathrm{Tr}(I-\rho_1)(I- Q)=\mathrm{Tr}I-\rho_1 - Q+\rho_1Q=\mathrm{Tr}\rho_1 Q$$
and so also $E(\mu_{\rho_1})|_{_W}=E(\mu_{I-\rho_1})|_{_W}$.

On the other hand, given a global section of $\underline{[0,\ln 2]}^\preceq$, the poset $\mathcal{V}(M_2)$ consists only of two-dimensional subalgebras. We can identify a context
$V=\{P, 1-P\}''$ for which $\gamma|_{_V}$ is minimal, and solve the equation $\mathrm{Sh}(x,1-x)=\gamma|_{_V}$ to find the eigenvalues of $\rho$. Since we have no further
information available, we cannot say which eigenvalue corresponds to which of the two projections generating $V$.

Note however that we are not far from reconstructing $\rho$: we would need to encode only one extra bit of information in order to fully reconstruct a two-dimensional quantum
state.

\subsection{A note on Gleason's Theorem}\label{gls}

Our result can be related to Gleason's theorem, since every finitely additive probability measures $\mu$ also determines a probability distribution $(\mu(P_1),\ldots,\mu(P_k))$ for each context $V=\{P_1,\ldots,P_k\}''$. Hence given $\mu$, we can define a map $E_\mu$ on the set of contexts by assigning to each context $V$ the Shannon entropy of its associated probability distribution. Using our reconstruction algorithm, we can get back the quantum state $\rho$ from $E_\mu$. Note however, that in order to do this we had to assume that we started from a probability measure $\mu$ on projections. Having an \emph{axiomatic} characterisation of those real-valued maps on contexts which are contextual entropy maps (and hence come from quantum states) would allow us to reconstruct quantum states directly.

\section{Other entropies}\label{other}

We have seen how Shannon entropy can be encoded in the topos approach, and how one can afterwards retrieve its quantum analogue, the von Neumann entropy. It is natural to ask at this point whether a similar encoding can be found for other classical entropies, and whether such an encoding would still enable us to retrieve their quantum analogues. We will look here at Renyi entropies, and show that it is possible to obtain their topos theoretic equivalent.

Renyi entropies form a one parameter family of Schur concave, additive entropies defined by 
$$R_q(p_1,\ldots,p_n)=\frac{1}{1-q}\ln\left[\sum_{i=1}^n p_i^q\right], \ \forall q\geq 0$$
Special cases of the Renyi entropies include $q=0$, which is the logarithm of the number of non-zero components of the distribution and is known as the \textit{Hartley entropy}. When $q\rightarrow 1$, we have the Shannon entropy, and when $q\rightarrow \infty$ the \textit{Chebyshev entropy} $R_\infty=-\ln p_{max}$, a function of the largest component $p_{max}$. 

For any given probability vector $\overrightarrow{p}$ the Renyi entropy is a continuous, non-increasing function of its parameter:
$$R_t(\overrightarrow{p})\leq R_q(\overrightarrow{q}), \ \forall t>q$$
To illustrate this, we have plotted in Figure \ref{Renyi} several Renyi entropies as functions of a probability distribution with two variables. Note that since Renyi entropies are Schur concave, their maximum value is attained for the totally mixed probability distribution, in which case $R_q(1/n,\ldots,1/n)=\ln n$.

\begin{figure}[H]
\centering
\includegraphics[width=8cm]{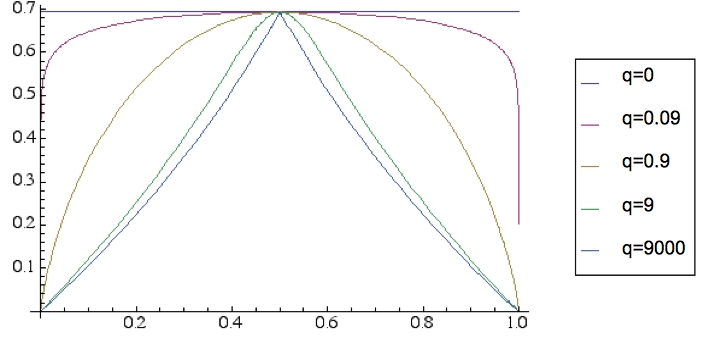}
\caption{$R_q(x,1-x)$ for $q=0, 0.5, 3$, $q\rightarrow 1$ and $q\rightarrow \infty$}\label{Renyi}
\end{figure}

At each parameter $q$, the quantum Renyi entropy can be defined on the set of density matrices as the classical Renyi entropy of the corresponding spectra:
$$\mathrm{R}_q(\rho)=\frac{1}{1-q}\ln\mathrm{Tr}(\rho^q)=\frac{1}{1-q}\ln\left[\sum_{i=1}^n \lambda_i^q\right]=R(\lambda_1,\ldots,\lambda_n)$$
Quantum Renyi entropy assigns the value $0$ to pure states exclusively, and $\ln n$ to the maximally mixed state $\rho_*=1/n I$.

We would like to define a contextual Renyi entropy using the same approach as in the case of the Shannon entropy. This suggests we should define contextual Renyi entropy locally as
$$\mathcal{R}_q(\mu)_{_V}=R_q(\mu|_{_V}(P_1),\ldots,\mu|_{_V}(P_n), \  \forall V=\{P_1,\ldots,P_n\}''$$
Of course, we would like these local components to fit together nicely as before, and to form a global section of some real number presheaf. For Shannon entropy, the fact that a global section could be formed was a consequence of the recursion property. Renyi entropies are in general not recursive, but they do satisfy a property which we shall call \textit{weak recursivity}, and we shall see that this is enough for our purposes. 

\begin{definition}
Let $\mathrm{S}$ be some function defined on the set of all probability distributions. If we coarse grain a probability distribution $(x_1,\ldots,x_n)$ by not distinguishing between all the outcomes, we obtain a new probability distribution with components
 $$p_1=\sum_{i=1}^{k_1} x_i, \  \ldots, \ p_r=\sum_{i=k_{r-1}+1}^{k_r} x_i$$
for some $0<k_1<k_2<\ldots<k_r=n$. We say that $\mathrm{S}$ is \textbf{weakly recursive} if
$$\mathrm{S}(x_1,\ldots,x_n)\geq \mathrm{S}(p_1,\ldots,p_r)$$
\end{definition}

One can easily check that Renyi entropies indeed satisfy this property, and hence for any two contexts $V'\supseteq V$
$$\mathcal{R}_q(\mu)_{_{V'}}\geq\mathcal{R}_q(\mu)_{_V}, \ \forall \mu\in\mathcal{M}(\underline{\Sigma})$$
This means it is possible to define contextual Renyi entropy as a mapping
$$\mathcal{R}_q:\mathcal{M}(\underline{\Sigma})\longrightarrow \Gamma \underline{[0,\ln n]^\preceq}$$

Since Renyi entropies are Schur concave, their quantum counterparts can be retrieved from the contextual Renyi entropies by finding the minimum over the set of values assigned to all maximal contexts. This is justified by the Schur-Horn lemma and similar arguments to those that were already used in Section \ref{VNeu}.

We will now briefly discuss some of the properties of Renyi entropies and their contextual analogues. 

\noindent\textbf{Concavity}

We saw in Section \ref{cns} that the global concavity of the contextual entropy was expressed as the concavity of each of its local components, and hence it was a direct consequence of the concavity property of Shannon entropy. Renyi entropies however are only concave for $0<q\leq 1$. In fact, it is known that concavity is lost for $q>q_*>1$, where $q_*$ depends on the dimension of the probability distribution. Concavity of the contextual Renyi entropies is then going to hold under the same conditions.

\noindent\textbf{Additivity and Subadditivity}

Renyi entropies are additive, so we can use the same justification as in Section \ref{cns} to defin/e subadditivity for contextual Renyi entropies as the following condition:
$$\mathcal{R}_q(\mu_\rho)_{_V}\leq\mathcal{R}_q(\mu_{\rho_1\otimes\rho_2})_{_V}, \ \forall V\in \mathcal{V}(\mathcal{B}(H))$$
This allows us to say that contextual Renyi entropies are also additive. On the other hand, since neither classical nor quantum Renyi entropies are subadditive (except for $q=0$ and $q=1$),  contextual Renyi entropy also doesn't have this property. 



\noindent\textbf{State reconstruction}

Finally, recall that the reconstruction algorithms described in Sections \ref{r1}-\ref{r3} relied on Gleason's theorem, the Schur-Horn lemma, and two extra ingredients: one was the fact that von Neumann entropy vanished only for pure states, and the second was the fact that for probability distributions with two variables one could find precisely two sollutions (symmetric around $1/2$) for which Shannon entropy would take any given value within its image. Both of these ingredients are present when we consider Renyi entropy, for positive parameter $q0$, as Figure \ref{Renyi} clearly indicates. This means that the reconstruction algorithms can also be applied to contextual Renyi entropies, with the exception of $\mathcal{R}_0$. 

\section{Summary and outlook}

Given a quantum state $\rho$ on a finite-dimensional Hilbert space $\mathcal{H}$ and a measurement context $V=\{P_1,\ldots,P_n\}''$, we can extract the probability distribution $(Tr(\rho P_1),\ldots,Tr(\rho P_n))$ by repeated preparations and measurements. In contrast to the quantum state itself, measurement contexts have direct operational meaning. The contextual entropy $E_\rho:\mathcal{V}\rightarrow [0,\ln n]$, assigns to each probability distribution its Shannon entropy and hence encodes data that can be extracted operationally from the quantum state $\rho$.

The fact that the state $\rho$ can be reconstructed from its contextual entropy $E_\rho$ if $\dim\mathcal{H}\geq 3$ provides a new, information-theoretic characterisation of quantum states that takes contextuality into account explicitly.  This characterisation can be generalized to other entropies with classical and quantum counterparts, such as R\'enyi entropies.

We presented a number of properties of the contextual entropy $E_\rho$ and discussed how the reconstruction of a quantum state from its contextual entropy relates to Gleason's theorem. As matters stand, the properties we have presented do not characterise contextual entropy fully: there are functions $F:\mathcal{V}\rightarrow [0,\ln n]$ that have all the properties discussed in the main text, but which are not the contextual entropy of any quantum state. An axiomatic characterisation of those functions which are contextual entropy maps promisses to be a non-trivial open question, as it would turn our reconstruction algorithm into an alternative proof of Gleason's theorem.

\vskip 4pt
\textbf{Acknowledgements.} We thank Oscar Dahlsten, Rui Soares Barbossa, Andrei Constantin, Samson Abramsky and Bob Coecke for discussions, and we thank Traian Abrudan for helpful technical explanations regarding the use of Matlab optimization codes. C.M.C. is supported by an EPSRC graduate scholarship.

\end{document}